\documentclass[10pt,a4paper]{article}
\usepackage{amssymb,amsmath,amsthm,bm}
\usepackage{graphicx}
\usepackage{geometry}
\usepackage{float}
\usepackage{tikz}
\usepackage{todonotes}
\graphicspath{{./figures/}}
\usetikzlibrary{decorations.pathreplacing}

\newtheorem{teo}            {Theorem}
\newtheorem{lema}     [teo]{Lemma}
\newtheorem{defin}[teo]{Definition}

\newtheorem{example} [teo]          {Example}
\newtheorem{obs} [teo]           {Remark}
\newtheorem{prop} [teo]       {Proposition}

\DeclareMathOperator{\Fix}{Fix}

\title{Wonderful models for generalized Dowling arrangements}
\author{Giovanni Gaiffi\footnote{Dipartimento di Matematica, Universit\`a di Pisa; \emph{E-mail address}:\texttt{gaiffi@dm.unipi.it}}, Viola Siconolfi\footnote{Dipartimento di Matematica, Universit\`a di Roma Tor Vergata; \emph{E-mail address}:\texttt{siconolf@mat.uniroma2.it}}}

\usepackage[font={small,bf}]{caption}
\usepackage{chngcntr}

\counterwithin{figure}{section}
\counterwithin{teo}{section}
\counterwithin{equation}{section}

\begin{document}
\maketitle
\begin{abstract}
For any triple given by a positive integer \(n\), a finite group \(G\), and a faithful representation \(V\) of \(G\),   one can describe a subspace arrangement whose intersection lattice is a generalized Dowling lattice in the sense of Hanlon \cite{Hanl}.
In this paper we  construct  the minimal  De Concini-Procesi wonderful model associated  to this subspace arrangement  and  give a description of its boundary.  Our aim is to point out   the nice poset  provided by the intersections of the irreducible components in the boundary, which   provides a geometric realization of the nested set poset of this generalized Dowling lattice.  It  can be represented by a family of  forests with leaves and labelings that depend   on the triple \( (n,G,V) \). We will  study it from the enumerative point  of view in the case when \(G\) is abelian.
\end{abstract}


\section{Introduction}
The De Concini-Procesi wonderful models  of subspace arrangements were  introduced in \cite{DCP1} and \cite{DCP2} and  play since then a crucial role in the study of configuration spaces and in various other fields of mathematics.
The relevance of their combinatorial properties and their relation with discrete geometry were pointed out   for instance in \cite{feichtner}, \cite{feichtnerkozlov}, \cite{postnikov}, \cite{GaiffiServenti2}, \cite{gaiffipermutonestoedra};   their relations with Bergman fans,  toric and tropical geometry were enlightened  in  \cite{feichtnersturmfels} and \cite{denham}; the connections between   the geometry of these models and   the Chow rings of matroids were  pointed out first in \cite{feichtneryuz} and  then in  \cite{adiprasitokatzhuh}, where they  also played a crucial role in the study of   some  log-concavity problems.

 Let us recall their definition: given a subspace arrangement $\mathcal{G}$ in \((\mathbb{C}^n)^*\), we consider for each $D\in\mathcal{G}$ its annihilator  $D^{\perp}$ in \(\mathbb{C}^n\) and the projective space $\mathbb{P}(\mathbb{C}^n/D^{\perp})$. Let $\mathcal{A(G)}$ be the complement of \(\bigcup_{D\in\mathcal{G}}D^\perp\)  in $\mathbb{C}^n$; we can then define the following embedding:
 $$i:\mathcal{A(G)}\rightarrow \mathbb{C}^n\times \prod_{D\in\mathcal{G}}\mathbb{P}(\mathbb{C}^n/D^{\perp}).$$ The wonderful model $Y_{\mathcal{G}}$ is defined as the closure of the image of $i$. If $\mathcal{G}$ is a building set (this is a combinatorial property that will be discussed in Section \ref{secbuilding}), $Y_{\mathcal{G}}$ turns out to be a smooth variety such that \(Y_{\mathcal{G}}-i(\mathcal{A(G)})\) is a divisor with normal crossings.

The geometric and topological properties of a wonderful model are  deeply connected with its initial combinatorial data. For instance, the poset of the intersections of the irreducible components in the boundary  \(Y_{\mathcal{G}}-i(\mathcal{A(G)})\) is   the {\em nested set} poset associated to  \(\mathcal G\) (see Sections 2 and 3). Moreover, the integer cohomology ring  of $Y_{\mathcal{G}}$ can be described using  some functions with integer values called 'admissible functions'  defined on    $\mathcal{G}$ (see \cite{DCP1}, \cite{YuzBasi}).

 The hyperplane arrangements associated to real and complex reflection groups give rise to particularly interesting  wonderful models that  have been widely studied in the literature. In the case of the arrangements of type $A_{n-1}$, for example, the minimal building set associated to its poset  of intersections produces a wonderful model which is isomorphic to the moduli space \(M_{0,n+1}\) of genus 0 stable \(n+1\)-pointed curves.

 In the case of complex reflection groups of  type $G(r,p,n)$ (according to the classification in \cite{Shep}), the lattice of intersections associated to the arrangement has the property of being a Dowling lattice. This is a combinatorial object defined in the early 70's by Dowling in \cite{dowling} through the action of a finite group $G$ on the set $G\times \{0, 1,\ldots,n\}$.  The corresponding minimal wonderful models have been studied for instance in \cite{hendersonwreath}, \cite{callegarogaiffi3}, \cite{Ga1}.
 
In the 90's, Hanlon introduced in \cite{Hanl} a generalized version of the Dowling lattices. These new objects are defined given a finite group, a family of its subgroups with a particular property ('closed subgroups') and a positive integer  $n$.  In particular,  given any finite group \(G\), any   faithful representation \(V\)  of \(G\) and any  positive integer \(n\),  one can  construct as in Section 3 of \cite{Hanl} a subspace arrangement \({\mathcal H}(n,G,V)\) in \(V^n\) whose intersection lattice is a generalized Dowling lattice.  We will call the arrangements \({\mathcal H}(n,G,V)\) {\em generalized Dowling arrangements}. 

 The question that motivated this paper is whether it is possible to give a concrete description of the minimal wonderful model associated to \({\mathcal H}(n,G,V)\), generalizing the case of complex reflection groups. 
In particular  our aim is to point out   the nice poset  provided by the intersections of the irreducible components in the boundary, which   provides a geometric realization of the nested set poset of the generalized Dowling lattice.  We observe that it  can be represented by a family of  forests with leaves and labelings that depend   on the triple \((n,G,V)\).  We will  study it from the enumerative point of view in the case when \(G\) is abelian, and we will  provide  formulas for the associated exponential generating series in Theorems \ref{teogstgenerale} and \ref{teogst}. 

This work is divided into six sections.
In Section 2 we recall the definition and the main properties of  the building sets  associated to a collection of subspaces, and of the nested sets associated to a building set.

In Section 3 we deal with  the De Concini-Procesi's wonderful model $Y_{\mathcal{G}}$ associated to a subspace arrangement $\mathcal{G}$; in particular  the section ends with a description of the boundary of $Y_{\mathcal{G}}$.

In Section 4, following Hanlon,  we define   the  subspace arrangement \({\mathcal H}(n,G,V)\) associated to a faithful representation \(V\) of a finite group $G$ and to a positive integer \(n\);  as we mentioned above, the lattice of its intersections, ordered by reverse inclusion,  turns out to be   a particular example of generalized Dowling lattice. We then study in Section 5 the minimal building set and the  nested set poset associated to \({\mathcal H}(n,G,V)\); in particular,  we define a bijection between the nested set poset   and  a family of labelled forests.

Finally in Section 6 we focus on the case when \(G\) is abelian and  we obtain some  formulas (see Theorems \ref{teogstgenerale} and \ref{teogst}) for the exponential generating series that  enumerates the nested sets and therefore the boundary components of the associated wonderful models. Part of the computation of these exponential formulas is completed in Section 7 (Theorem \ref{formulaHtrees}).

\section{Building and nested sets}
\label{secbuilding}
 Given a complex space $V$ and its dual space $V^*$, we consider a finite set $\mathcal{H}$ of subspaces in $V^*$. We denote by $\mathcal{C_H}$ the set of subspaces in $V^*$ that can be written as sum of elements in $\mathcal{H}$.
 \begin{defin}
The set  $\mathcal{G}$ of subspaces in $V^*$ is {\em  building}  if each $C\in\mathcal{C_G}$ can be written as $C=G_1\oplus\ldots\oplus G_s$ where the $G_i$s are the maximal elements in $\mathcal{G}$ among the ones that are contained in $C$.
 \end{defin}
The nested sets associated to a given building set are defined  as follows:
\begin{defin}
Given a building set $\mathcal{G}$, a subset $\mathcal{S}\subseteq \mathcal{G}$ is said to be ($\mathcal{G}$-)nested if  for any integer \(p\geq 2\) and for any choice of  $S_1,\ldots,S_p\in\mathcal{S}$ not comparable with respect to inclusion, then  $S_1,\ldots,S_p$ are in direct sum and $S_1\oplus\ldots\oplus S_p\notin \mathcal{G}$.
\end{defin}

Given any  finite set $\mathcal{H}$ of subspaces in $V^*$, the set $\mathcal{C_H}$ defined before is building. The following definition allows us to find another building set naturally associated to $\mathcal{H}$. 

\begin{defin}
Given a set $\mathcal{H}$ of subspaces in $V^*$
 and a subspace $U\in \mathcal{C_H}$, a decomposition of $U$ is a list of non zero subspaces $U_1,U_2,\ldots,U_k\in \mathcal{C_H}$ with $U=U_1\oplus\ldots\oplus U_k$ such that for every subspace $A\subseteq U$ in $\mathcal{C_H}$, also $A\cap U_1$, $A\cap U_2,\ldots,A\cap U_k$ lie in $\mathcal{C_H}$ and $A=(A\cap U_1)\oplus(A\cap U_2)\oplus\ldots\oplus (A\cap U_k)$.
\end{defin}
If a subspace does not admit a decomposition it is called irreducible. The set of irreducible subspaces of $\mathcal{C_H}$ is denoted $\mathcal{F_H}$ and it turns out to be a building set. It is the minimal building set that contains $\mathcal{H}$ and is  contained in $\mathcal{C_H}$.

\section{Wonderful models}

Wonderful models have been introduced by De Concini and Procesi in their papers \cite{DCP1} and \cite{DCP2}. Here we recall briefly how such varieties are defined.

Let $V$ be a finite dimensional vector space over  $\mathbb{C}$. Let us consider a finite family $\mathcal{G}$ of subspaces of $V^{*}$, and for every $A\in \mathcal{G}$ consider its annihilator $A^{\perp}\subseteq V$. 
 We also denote by $\mathbb{P}_A$ the projective space of lines in $V/A^{\perp}$.

Let $\mathcal{V}_{\mathcal{G}}:=\cup_{\mathcal{A}\in\mathcal{G}}A^{\perp}$ be the union of all the subspaces $A^{\perp}$ and $\mathcal{A}_{\mathcal{G}}$ be the complement of $\mathcal{V}_{\mathcal{G}}$ in $V$. The rational map
\[
\pi_A:V\rightarrow V/A^{\perp}\rightarrow \mathbb{P}_A
\]
is defined outside $A^{\perp}$ and thus we have a regular morphism $\mathcal{A}_{\mathcal{G}}\rightarrow \prod_{A\in \mathcal{G}}\mathbb{P}_A$. The graph of this morphism is a closed subset of $\mathcal{A}_{\mathcal{G}}\times \prod_{A\in \mathcal{G}}\mathbb{P}_A$ which embeds as an open set into $V\times \prod_{A\in \mathcal{G}}\mathbb{P}_A$. Finally we have an embedding
\[
\rho: \mathcal{A}_{\mathcal{G}}\longrightarrow V\times \prod_{A\in \mathcal{G}}\mathbb{P}_A
\]
as a locally closed subset.

\begin{defin}
The wonderful model $Y_{\mathcal{G}}$ associated to $\mathcal{G}$ is 
    the closure of the image of $\rho$.
\end{defin}

In section $3$ of \cite{DCP1} De Concini and Procesi prove  that if $\mathcal{G}$ is a building set then $Y_{\mathcal{G}}$ is a smooth variety. They also provide  the following description of the boundary of $Y_{\mathcal{G}}$:
\begin{teo}
\begin{itemize}
\item[(1)] The complement $D$ of $\mathcal{A_G}$ in $Y_{\mathcal{G}}$ is the union of some smooth irreducible divisors $D_G$ indexed by the elements $G\in\mathcal{G}$. In particular  $D_G$ is the unique irreducible component in $D$ such that $\delta(D_G)=G^{\perp}$ (where  $\delta$ is the projection map $\delta:Y_{\mathcal{G}}\rightarrow V$).
\item[(2)] The divisors $D_{A_1},\ldots, D_{A_n}$ have nonempty intersection if and only if the set $\{A_1,\ldots,A_n\}$ is $\mathcal{G}$-nested. In this case their intersection is transversal and irreducible.
\end{itemize}
\end{teo}

As a consequence of this theorem, we notice that   the  poset provided by the intersections of the irreducible components in the boundary ordered by reverse inclusion  is isomorphic to the poset  of the $\mathcal{G}$-nested sets ordered by inclusion. 

\section{Generalized Dowling arrangements}
We start considering a finite group $G$,   a positive integer   $n$,  and a finite dimensional complex faithful representation $\rho:G\rightarrow GL(V)$  which does not contain the trivial representation. For each subgroup $H$ of $G$ we denote by $\Fix(H)=\{v\in V:\rho(h)v=v\text{ for all }h\in H\}$ the set of its fixed points in \(V\).  Notice that  since \(V\) does not contain the trivial representation we have  \(\Fix(G)=\{0\}\).

Let us now denote by \({\mathcal P}(G)\) the poset (with respect to inclusion) of  all the subgroups of $G$.  We  define  a function \(\phi\: : \: {\mathcal P}(G)\rightarrow {\mathcal P}(G)\)  requesting  that,  for every subgroup \(H\),  $\phi({H})$ is the maximal subgroup, with respect to inclusion,  such that $\Fix(\phi(H))= \Fix(H)$. 

We  denote by $\mathcal{K}$ the set of subgroups fixed by \(\phi\), which we will call {\em closed subgroups} following Hanlon's terminology in \cite{Hanl}. We consider $\mathcal{K}$ as a poset with respect to inclusion: we notice that  \(G\)  is a closed subgroup since it is maximal with respect to inclusion and  $\{e\}$ is a closed subgroup since  the representation \(V\) is  faithful.  We also point  out that  for 
\(K, K' \in  \mathcal{K}\)  we have \(K=K'\) if and only if $\Fix(K) = \Fix(K')$.

  Let us then consider the subspace arrangement  \({\mathcal H}(n,G,V)\) given by the following subspaces in $V^n$:
 \[
 H(i, j, g)=\{(v_1,\ldots,v_n)|v_{j}=\rho(g)v_i\}  
 \]
 for \(1\leq i<j\leq n\) and \(g\in G\), and 
\[
H(i,i,g)=\{(v_1,\ldots,v_n)|v_i=\rho(g)v_i\}
\]
 for \(1\leq i \leq n\) and \(g\in G-\{e\}\).

We notice that for \(i\neq j\) the subspace \(H(i, j, g)\) has codimension \(dim \ V\) in \(V^n\) while the subspace \(H(i, i, g)\) has codimension \(dim \ V - dim \ Fix(<g>)\), where \(<g>\) is the cyclic subgroup generated by \(g\).

We call  \({\mathcal H}(n,G,V)\) the  {\em generalized Dowling arrangement} associated to the triple \((n,G,V)\).

The motivation for this name comes from the following remark: let  $L(n, G, V)$ be its intersection  lattice, ordered by reverse inclusion. It  provides  an example  of  the   generalized Dowling lattices introduced by Hanlon, as it was shown  in Section 3 of  \cite{Hanl}. More precisely,  using Hanlon's terminology, the function \(\phi\) is a {\em closure operator} and $L(n, G, V)$ is isomorphic to  $D_n(G, K(\phi))$.

In particular, when \(G\) is \(\mathbb{Z}_r\) and \(V\) is an irreducible (one dimensional) representation,  the lattice $L(n, G, V)$  is  isomorphic to the Dowling lattice \(Q_n(\mathbb{Z}_r)\). It can be seen as  the intersection lattice of the hyperplane arrangement associated to the complex reflection group \(G(r,1,n)\).

\section{Nested sets and labelled forests}

After fixing the triple \((n, G, V)\), let us denote for brevity  by \(\mathcal L'\)  the intersection lattice $L(n, G, V)$ described in the preceding section.

We define $\mathcal{G}'$ to be the subset of \(\mathcal L'\) whose elements are the subspaces:

\begin{itemize}
\item \[
H^{K}(i_{1}^{g_{1}K},\ldots,i_{k}^{g_{k}K}):=\{(v_{1},\ldots,v_{n})\in V^n|v_{i_1}=g_1v,\ldots,v_{i_k}=g_kv,v\in \Fix(K)\},
\] 
where \(1< k\leq n\), 
$1 \leq i_1<\ldots < i_k\leq n$,
$K$ ranges in the set \(\mathcal K\) of closed  subgroups of $G$,  the \(g_jK\)'s are cosets and the symbol \(i_j^{g_{j}K}\) represents the  ordered pair given by the index \(i_j\) and the coset \(g_jK\);

\item \[
H^{K}(i_{}^{eK}):=\{(v_{1},\ldots,v_{n})\in V^n|v_{i}\in \Fix(K)\},
\] where $1 \leq i \leq n$ and  $K$ ranges in the set \(\mathcal K\) of closed  subgroups of $G$ different from \(\{e\}\).
\end{itemize} 
 We notice that here (and from now on) we omitted \(\rho\) for brevity, using the notation  \(gv\) instead of \(\rho(g)(v)\).
It is easy to check that the intersection lattice of $\mathcal{G}'$ is equal to \(\mathcal L'\); more precisely, every element \(A\) of \(\mathcal L'\) can be expressed as the transversal intersection of the minimal elements of $\mathcal{G}'$ containing \(A\).

Therefore, if we  denote by  $\mathcal{G}$ the set of the annihilators in \((V^n)^*\) of the elements in $\mathcal{G}'$, we have that \(\mathcal{C}_{\mathcal{G}}\) is the set of the annihilators of the elements in \(\mathcal L'\) and  $\mathcal{G}$ is the building set of irreducibles of  \(\mathcal{C}_{\mathcal{G}}\).

Let us now study the nested sets associated to the building set $\mathcal{G}$. We begin by stating two simple lemmas.
\begin{lema}
Given a closed subgroup $H<G$, all its conjugate subgroups are closed as well.
\end{lema}
\begin{proof}
Let us suppose that $H'=gHg^{-1}$ is not a closed subgroup. We know that there is a $K\supsetneq H'$ such that $K$ is closed and $\Fix(K)=\Fix(H')$. Hence we have $g^{-1}Kg\supsetneq g^{-1}H'g=H$ and $\Fix(H)=g^{-1}\Fix(H')=g^{-1}\Fix(K)=\Fix(g^{-1}Kg)$ and this is absurd.
\end{proof}
 We now have the main property for the elements in $\mathcal{G}$:
\begin{lema}\label{lemav1}
Given $K< G$ a subgroup and $K^g=gKg^{-1}$,
\[
H^K(i_1^{gK},i_2^{g_2K},\ldots,i_k^{g_kK})=H^{K^g}(i_1^{eK^g},i_2^{g_2g^{-1}K^g},\ldots,i_k^{g_kg^{-1}K^g})
\]
holds.
\end{lema}
\begin{proof}
Let $w\in H^K(i_1^{gK},i_2^{g_2K},\ldots,i_k^{g_kK})$, $w=(v_1,\ldots,v_n)$, this implies $v_{i_1}=gv$, $v\in\Fix(K)$ and $v_{i_j}=g_jv$ for $j=2,\ldots,n$. So $v_{i_1}\in \Fix(K^g)$ and $v_{i_j}=g_jg^{-1}v_{i_1}$, these are the conditions for $v\in H^{K^g}(i_1^{eK^g},i_2^{g_2g^{-1}K^g},\ldots,i_k^{g_kg^{-1}K^g})$.
\end{proof}

Thanks to this lemma we can write  $$\mathcal{G}'= \{H^K(i_1^{eK},i_2^{g_2K},\ldots,i_k^{g_kK})\ | \ \hbox{if} \ k=1  \ K \in \mathcal K-\{e\}, \ \hbox{if} \ k>1 \   \ K \in \mathcal K \}.$$

Given two elements in $\mathcal{G}'$, $H^1=H^{K}(i_1^{eK},\ldots,i_k^{g_{k}K})$ and $H^2=H^{K'}(j_1^{eK'},\ldots,j_t^{h_{t}K'})$, we want to determine when there exists a nested set containing both their annihilators.

\begin{prop}\label{propv1}
 Given $H^1$ and $H^2$ as above, then $H^1\supseteq H^2$ if and only if the following facts hold:
    \begin{itemize}
    \item[(1)] $\{i_1,\ldots,i_k\}\subseteq \{j_1,\ldots,j_t\}$,
    \item[(2)] for all the $p\in \{i_1,\ldots,i_k\}$ (say \(p=i_r=j_t\) in view of point \((1)\)),   we have  $\Fix(K)\supseteq g_r^{-1}h_t\Fix(K')$.
    \item[(3)] for all $p,s\in \{i_1,\ldots,i_k\}$ (say \(p=i_r=j_t\), \(s=i_\theta=j_\gamma\)  in view of point \((1)\)), we have $h_\gamma^{-1}g_\theta g_r^{-1}h_t\in K'$.
    \end{itemize}
\end{prop}

\begin{proof}
We start supposing $H^2\subseteq H^1$.
\begin{itemize}
\item[(1)] Obvious.
\item[(2)] Let us take $p\in \{i_1,\ldots,i_k\}$, say \(p=i_r=j_t\).  Then $v\in H_2\Rightarrow v\in H_1$ so if we  have $v_p=v_{j_t}=h_t w',$ with $w'\in \Fix(K')$  we also have  $v_p=v_{i_r}=g_r w,$ with $w\in\Fix(K)$. This implies $\Fix(K)\supseteq g_r^{-1}h_t\Fix(K')$.
\item[(3)] Given $p,s\in \{i_1,\ldots,i_k\}$ with  \(p=i_r=j_t\), \(s=i_\theta=j_\gamma\), we must have $g_rg_\theta^{-1}v_s=v_p$ and $h_th_\gamma^{-1}v_s=v_p$, so $(g_rg_\theta^{-1})^{-1}h_th_\gamma^{-1}v_s=v_s$. Since $v_s$ can be any element in $h_\gamma\Fix(K')=\Fix(h_\gamma K'h_\gamma^{-1})$ we deduce that $(g_rg_\theta^{-1})^{-1}h_th_\gamma^{-1}\in h_\gamma K'h_\gamma ^{-1}$ and so $h_\gamma^{-1}g_\theta g_r^{-1}h_t\in K'$.
\end{itemize}
Suppose now that the three conditions are true. Taken $v\in H^2$ we want to show that $v\in H^1$. We have $v=(v_1,\ldots,v_n)$ and for all $s=j_\gamma \in \{j_1,\ldots,j_t\}$, $v_s=h_\gamma w',$ with $w'\in\Fix(K')$. Thanks to conditions $(1)$ and $(2)$ we have that for all $p=i_r \in \{i_1,\ldots,i_k\}$, $g_r^{-1}v_p=w$ with
 $w\in\Fix(K)$. Now we have to check that for all $p=i_r=j_t, s=i_\theta=j_\gamma \in \{i_1,\ldots,i_k\}$ we have $g_r^{-1}v_p=g_\theta^{-1}v_s$. Now  we have \(v_p=h_tw', v_s=h_\gamma w'\) with \(w'\in Fix \ K'\), therefore   $g_r^{-1}v_p=g_r^{-1}h_tw'$ while $g_\theta^{-1}v_s=g_\theta^{-1}h_\gamma w'$. The equality follows because
 $h_\gamma^{-1}g_\theta g_r^{-1}h_t\in K'$ (condition $(3)$). 
\end{proof}

So if the annihilators  of $H^1$ and $H^2$ belong to the same \(\mathcal G\)-nested set then one of the following cases holds:
\begin{itemize}
 \item $H^1$ and $H^2$ are  one included into the other, and  the three conditions of Proposition \ref{propv1} hold.
  \item $\{i_1,\ldots,i_k\}\cap \{j_1,\ldots,j_t\}=\emptyset$  and $K\neq G$ or $K'\neq G$.  In this case the annihilator of $H^1\cap H^2$ is not an element of the building set $\mathcal{G}$.
\end{itemize}
In the case neither of the two conditions above holds we have that  $H^1\cap H^2 \in \mathcal G'$ so its annihilator belongs to  $\mathcal{G}$.

Now we need to  recall the definition of the poset of conjugacy classes of subgroups in a group \(G\).
\begin{defin}
Given a subgroup \(H\) of \(G\), we denote by \([H]\) its conjugacy class. The set \(\Lambda(G)\) of conjugacy classes of \(G\) is a poset with the following order relation:
\[ [P]\leq [Q] \quad \hbox{iff} \quad \exists g\in G \; \hbox{such that } \; gPg^{-1}\subseteq Q\]
\end{defin}
Since the set \(\mathcal K\) of closed subgroups is closed with respect to conjugation, the set of conjugacy classes in \(\mathcal K\) can be viewed as a  subposet of \(\Lambda(G)\).
If \(G\) is abelian, this subposet coincides with \(\mathcal K\) ordered by inclusion.

\begin{obs}
\label{obsconjugacy}
We notice that in condition \((2)\) of Proposition \ref{propv1},  $\Fix(K)\supseteq g_r^{-1}h_t\Fix(K')$ is equivalent to $\Fix(K)\supseteq \Fix(g_r^{-1}h_tK'h_t^{-1}g_r)$. Since the involved subgroups are closed,  this in turn is equivalent to \(K\subseteq g_r^{-1}h_tK'h_t^{-1}g_r\), that is to say \(h_t^{-1}g_rKg_r^{-1}h_t\subseteq K'\), that  in particular implies  \([K]\leq [K']\).
\end{obs}

As a consequence of Proposition \ref{propv1} and of the  Remark \ref{obsconjugacy} above,  the collection of all the $\mathcal{G}$-nested sets is in bijection with the following  family of forests.

We start considering $n$ vertices labelled with the elements in $M=\{1,\ldots,n\}$ and a partition of $M$ into  some sets $M_1,\ldots,M_k$. We now consider for each $M_i$ a rooted oriented tree $\mathcal{T}_i$ whose leaves are labelled by the elements of $M_i$ (so in the end we will have a forest $\mathcal{F}$ with $k$ trees). Then we  label each internal vertex (i.e. each vertex different from a leaf) and each edge of $\mathcal{T}_i$ according to the following rules:
\begin{itemize}\label{nonbif}
 \item[(1)] Each internal vertex  is labelled with a closed subgroup of $G$. If \(v\) is labelled by \(Q\) and \(w\) is an internal vertex descendant of \(v\) labelled by \(P\), then \([P]\leq [Q]\).
  \item[(2)] If an internal  vertex \(w\) labelled by the closed subgroup \(P\) is the only direct descendent of a vertex \(v\)  labelled by the closed subgroup \(Q\) then \([P]\lneq [Q]\). 
  If a  leaf    is the only direct descendent of a vertex \(v\),  the label of \(v\)  is different from \(\{e\}\).
   \item[(3)] The label $G$  appears in at most one tree of $\mathcal{F}$, and  a  vertex labelled with  \(G\) has at most one  direct descendant labelled with $G$.
  \item[(4)] If an edge stems from a vertex, it is labelled with a coset of the closed subgroup labeling the vertex. More precisely, if \(w\), labelled by \(P\),  is a direct descendant of \(v\), labelled by \(Q\), then the label of the edge with vertices \(v\) and \(w\) is \(aQ\), with \(a\) such that \(a^{-1}Pa\subseteq Q\).
   \item[(5)] For an internal vertex \(v\) labelled with \(K\), the edge connecting $v$ to the subtree containing the
smallest leaf is labelled with the coset \(eK\).

\end{itemize}
The following picture is an example of  a forest  associated to a nested set; here we consider  a group \(G\) with  some subgroups  \(\Gamma_1, \Gamma_2, \Gamma_1'\) such that \(\{e\}\subsetneq \Gamma_2\subsetneq \Gamma_1\subsetneq G\) and \(\{e\}\subsetneq \Gamma_2\subsetneq \Gamma_1'\subsetneq G\):

\begin{figure}[H]  
\centering
\begin{tikzpicture}[scale=1.8]
     \draw (0,0) node[anchor=north]  {\footnotesize4}; 
     \draw (0.5,0) node[anchor=north]  {\footnotesize7};
     \draw (1.5,0) node[anchor=north]  {\footnotesize3};
     \draw (2.25,0) node[anchor=north]  {\footnotesize6};
     \draw (2.75,0) node[anchor=north]  {\footnotesize2};
     \draw (3.5,0) node[anchor=north]  {\footnotesize1};
     \draw (4,0) node[anchor=north]  {\footnotesize5};
     \draw (5,0) node[anchor=north]  {\footnotesize8}; 
     \draw (0.25,0.5) node[anchor=east]  {\footnotesize{$\Gamma_1$}}; 
     \draw (2.5,0.5) node[anchor=west]  {\footnotesize{$\Gamma_2$}}; 
     \draw (2.25,1) node[anchor=west]  {\footnotesize{$\Gamma_2$}};  
     \draw (4,1) node[anchor=west]  {\footnotesize{$\Gamma_2$}};     
     \draw (1.5,1.5) node[anchor=west]  {\footnotesize{$\Gamma_1$}}; 
     \draw (1.5,2) node[anchor=west]  {\footnotesize{$G$}};  
     \draw (4,2) node[anchor=west]  {\footnotesize{$\Gamma'_1$}}; 
     \draw (0.12,0.25) node[anchor=east]  {\footnotesize{$e \Gamma_1$}}; 
     \draw (0.37,0.25) node[anchor=west]  {\footnotesize{$b\Gamma_1$}}; 
     \draw (2.37,0.25) node[anchor=east]  {\footnotesize{$d\Gamma_2$}}; 
     \draw (2.62,0.25) node[anchor=west]  {\footnotesize{$e\Gamma_2$}}; 
     \draw (1.85,0.5) node[anchor=east]  {\footnotesize{$c\Gamma_2$}}; 
     \draw (0.9,1.05) node[anchor=east]  {\footnotesize{$a\Gamma_1$}}; 
     \draw (1.95,1.25) node[anchor=west]  {\footnotesize{$e\Gamma_1$}}; 
     \draw (2.37,0.75) node[anchor=west]  {\footnotesize{$e\Gamma_2$}}; 
     \draw (1.5,1.75) node[anchor=west]  {\footnotesize{$G$}}; 
     \draw (4,1.5) node[anchor=west]  {\footnotesize{$e\Gamma'_1$}}; 
     \draw (4,0.5) node[anchor=west]  {\footnotesize{$e\Gamma_2$}}; 

\foreach \Point in {(0,0), (0.5,0),(0.25,0.5),(1.5,1.5), (1.5,2) , (1.5,0),(2.25,1),(2.25,0),(2.5,0.5),(2.75,0),(3.5,0),(4,0),(4,1),(4,2),(5,0)}{
    \node at \Point {\textbullet};
}
    \draw[thick] (0.25,0.5) -- +(0.25,-0.5); 
    \draw[thick] (0.25,0.5) -- +(-0.25,-0.5); 
    \draw[thick] (0.25,0.5) -- +(1.25,1); 
    \draw[thick] (1.5,1.5) -- +(0,0.5); 
    \draw[thick] (2.25,1) -- +(-0.75,0.5); 
    \draw[thick] (2.25,1) -- +(+0.25,-0.5); 
    \draw[thick] (2.25,1) -- +(-0.75,-1); 
    \draw[thick] (2.5,0.5) -- +(-0.25,-0.5); 
    \draw[thick] (2.5,0.5) -- +(0.25,-0.5); 
    \draw[thick] (4,2) -- +(0,-1); 
    \draw[thick] (4,1) -- +(0,-1); 
  \end{tikzpicture}
\caption{}
\label{foresta1}
\end{figure}

where $a,b,c,d\in G$.

We will denote by $\mathcal{F}(n,G,V)$ the set of all the forests (with at least a non-leaf vertex) constructed as above.
The bijection between  this set and  the family of nested sets is obtained as follows:  given a forest one constructs its corresponding nested set by associating  to every internal vertex \(v\) the subspace 
\(H^K(i_1^{g_1K},...,i_j^{g_jK})\) where 
\begin{itemize}
\item \(i_1,...,i_j\) are the labels of the leaves descending from \(v\);
\item \(K\) is the label of \(v\);
\item if in the path from \(v\) to the leaf \(i_r\) one finds (from the top) the cosets \(a_1K, a_2K_2,...,a_lK_l\) then \(g_rK=a_l\cdots a_2a_1K\). 
\end{itemize}
For example the forest in Figure \ref{foresta1} corresponds to the nested set provided by the following seven subspaces:  
\[
\begin{array}{l}
 H^{G}(2^{G},3^{G},4^{G},6^{G},7^{G}),  \\
H^{\Gamma_1}(2^{e\Gamma_1},3^{c\Gamma_1},4^{a\Gamma_1},6^{d\Gamma_1},7^{ba\Gamma_1}), \\
H^{\Gamma_1}(4^{e\Gamma_1},7^{b\Gamma_1}), \\
H^{\Gamma_2}(2^{e\Gamma_2},3^{c\Gamma_2},6^{d\Gamma_2}),\\
H^{\Gamma_2}(2^{e\Gamma_2},6^{d\Gamma_2}),\\
H^{\Gamma_1'}(5^{e\Gamma_1'}),\\
  H^{\Gamma_2}(5^{e\Gamma_2}).
\end{array}
\]

\begin{obs}
When \(G\) is abelian, the description of the forests in $\mathcal{F}(n,G,V)$ is simplified, since the conjugacy classes of subgroups are singletons.
More precisely, points \((2)\) and \((4)\) can be substituted by the following points:
\begin{itemize}\label{abeliannonbif}
  \item[(2)'] If an internal  vertex \(w\) labelled by the closed subgroup \(P\) is the only direct descendent of a vertex \(v\)  labelled by the closed subgroup \(Q\) then \(P\subsetneq Q\). 
  If a  leaf    is the only direct descendent of a vertex \(v\),  the label of \(v\)  is different from \(\{e\}\).
  \item[(4)'] If an edge stems from a vertex, it is labelled with a coset of the subgroup labeling the vertex. 

\end{itemize}

\end{obs}
In the next section we will provide a formula for the exponential  series that enumerates such forests in the case when \(G\) is abelian. 
\begin{obs}
The  nested sets are the  main ingredient  in the description of a basis of the cohomology of the corresponding De Concini-Procesi model.  As we mentioned in the Introduction, some special functions, the {\em admissible functions}, play a relevant role in this description.  As a further research, one could  investigate if the methods used in this paper can be adapted to find formulas  for the generating series of the Betti numbers. 
\end{obs}

\section{Enumeration of the nested sets in the case when \(G\) is abelian}

In this section and in the next one we will assume \(G\) abelian.

Given an oriented  forest \(F\) in $\mathcal{F}(n,G,V)$, we associate to it a set of subforests, selected according to the labels of its vertices, that are  closed subgroups of \(G\). Let us consider a closed subgroup \(H\) that appears among the labels of the vertices of \(F\). The subforest associated to \(H\) is provided by all the   vertices of \(F\) labelled by \(H\) and by all the labelled edges that go out (according to the orientation) from these vertices. If the other end of one of these edges is (in \(F\)) a vertex \(w\) which is labelled by a subgroup different from \(H\) or is a leaf, then \(w\) will be an unlabelled leaf of our subforest.
We also consider as a subforest the set of the leaves that are disconnected from the rest of the forest (we will call them {\em fallen leaves}). For instance in the Figure \ref{foresta1} the leaves 1 and 8 are fallen leaves
and  the  decomposition into subforests is the following:
\begin{figure}[H]
\centering
  \begin{tikzpicture}[scale=1]

     \draw (-0.5,1) node[anchor=east]  {\footnotesize{$\Gamma_1$}}; 
     \draw (0,2) node[anchor=south]  {\footnotesize{$\Gamma_1$}}; 
     \draw (-0.75,0.5) node[anchor=east]  {\footnotesize{$e\Gamma_1$}}; 
     \draw (-0.25,0.5) node[anchor=west]  {\footnotesize{$b\Gamma_1$}}; 
     \draw (-0.25,1.5) node[anchor=east]  {\footnotesize{$a\Gamma_1$}}; 
     \draw (0.25,1.5) node[anchor=west]  {\footnotesize{$e\Gamma_1$}}; 
     \draw (2.5,2) node[anchor=south]  {\footnotesize{$e\Gamma_1'$}}; 
     \draw (2.5,1) node[anchor=east]  {\footnotesize{$\Gamma_1'$}}; 
     \draw (5,2) node[anchor=south]  {\footnotesize{$\Gamma_2$}}; 
     \draw (5.5,1) node[anchor=west]  {\footnotesize{$\Gamma_2$}}; 
     \draw (7,2) node[anchor=south]  {\footnotesize{$\Gamma_2$}}; 
     \draw (4.5,1) node[anchor=east]  {\footnotesize{$c\Gamma_2$}}; 
     \draw (6.05,1.5) node[anchor=east]  {\footnotesize{$e\Gamma_2$}}; 
     \draw (5.25,0.5) node[anchor=east]  {\footnotesize{$d\Gamma_2$}}; 
     \draw (5.75,0.5) node[anchor=west]  {\footnotesize{$e\Gamma_2$}}; 
     \draw (7,1) node[anchor=east]  {\footnotesize{$e\Gamma_2$}}; 
     \draw (8.5,1) node[anchor=east]  {\footnotesize{$G$}}; 
     \draw (8.5,2) node[anchor=south]  {\footnotesize{$G$}}; 
     \draw (10.5,1) node[anchor=south]  {\footnotesize{}}; 
     \draw (11,1) node[anchor=south]  {\footnotesize{}}; 
\foreach \Point in {(-1,0), (0,0),(-0.5,1),(0,2), (0.5,1) ,
(2.5,0),(2.5,2),(4,0),(5,0),(5,2),(5.5,1),(6,0),(7,0),(7,2),(8.5,0),(8.5,2),(10.5,0),(11,0)}{
    \node at \Point {\textbullet};
}
    \draw[thick] (-0.5,1) -- +(-0.5,-1); 
    \draw[thick] (-0.5,1) -- +(0.5,-1); 
    \draw[thick] (-0.5,1) -- +(0.5,1); 
    \draw[thick] (0,2) -- +(0.5,-1); 
    \draw[thick] (2.5,0) -- +(0,2); 
    \draw[thick] (5.5,1) -- +(-0.5,-1); 
    \draw[thick] (5.5,1) -- +(0.5,-1); 
    \draw[thick] (5.5,1) -- +(-0.5,1); 
    \draw[thick] (5,2) -- +(-1,-2); 
    \draw[thick] (7,0) -- +(0,2); 
    \draw[thick] (8.5,0) -- +(0,2); 
  \end{tikzpicture}
\caption{}
\end{figure}

Given a closed subgroup  $H\neq G$, let us now  focus on the \(H\)-forests, i.e. the forests whose  internal vertices have the same label \(H\); in particular we are interested in counting all the forests of this type whose \(n\) leaves are labelled by a fixed set of \(n\) numbers.

 We will call $H$-tree a rooted oriented  tree whose internal vertices  are labelled with the subgroup $H$, while the edges are labelled with  cosets of $H$ and the leaves are labelled by a fixed set of  numbers.

 Let us denote by $\lambda_H(t_H)$ the exponential series that counts the number of possible $H$-trees with labelled leaves:
 \begin{equation}\label{gamma}
\lambda_H(t_H)=\sum_{i\geq 1}\frac{\lambda_{H,i}t_H^i}{i!}
\end{equation}
where $\lambda_{H,i}$ is the number of $H$-trees with $i$ leaves. A formula for this series will be shown in Section \ref{secHtrees}.

Suppose    that we want to compute the number of forest with three $H$-trees, having respectively $l_1$, $l_2$ and $l_3$ leaves, with \(l_1>l_2>l_3\). This is
\[
\binom{l_1+l_2+l_3}{l_1}\lambda_{H,l_1}\binom{l_2+l_3}{l_2}\lambda_{H,l_2}\lambda_{H,l_3}.
\]
The binomial coefficients are needed to assign the $l_1+l_2+l_3$ labels to the three trees. We then have
  \[
  \lambda_{H,l_1}\frac{t_H^{l_1}}{l_1!} \lambda_{H,l_2}\frac{t_H^{l_2}}{l_2!} \lambda_{H,l_3}\frac{t_H^{l_3}}{l_3!}=
  \lambda_{H,l_1}\lambda_{H,l_2}\lambda_{H,l_3}\binom{l_1+l_2+l_3}{l_1}\binom{l_2+l_3}{l_2}\frac{t_H^{l_1+l_2+l_3}}{(l_1+l_2+l_3)!}.
  \]
This elementary observation can be easily generalized to show that    the   $H$-forests with three $H$-trees are counted by the series  $\frac{(\lambda_H(t_H))^3}{3!}$. 

Therefore  the contribution given by the forests whose internal vertices are labelled with $H$  is
\[
\sum_{n\geq 0} \frac{(s\lambda_H(t_H))^n}{n!}=e^{s\lambda_H(t_H)}.
\]
Here we also added the new variable $s$ that counts the number of connected components (trees) of the forest: the coefficient of the monomial $\frac{s^kt_H^j}{j!}$ is the number of forests with $k$ trees and $j$ leaves.

\medskip

Let us now focus on the following question: how many are the forests in $\mathcal{F}(n,G,V)$ whose decomposition into subforests is given by an  $H$-tree with $j$ leaves and a $K$-tree with $i$ leaves, where \(H\subsetneq K\)?

We know  that there are $\lambda_{K,i}$ different $K$-trees with $i$ leaves and $\lambda_{H,j}$ $H$-trees with $j$ leaves,  then we have  the following cases:
\begin{itemize}
  \item the forest has two trees;  the number of such forests is $\lambda_{K,i}\lambda_{H,j}\binom{i+j}{i}$;
  \item the forest is given by a single tree. We can think of this forest as if  it was obtained by glueing   the root of an $H-$tree with $j$  leaves with one of the leaves  of a  $K-$tree with \(i\) leaves. In this case the final number of leaves of the forests is $i+j-1$ and the number of all these forests is $\binom{i+j-1}{j}\lambda_{K,i}\lambda_{H,j}$.
\end{itemize}
Now we observe that
\[
(s+\frac{\partial}{\partial t_K})\frac{\lambda_{K,i}t_K^i}{i!}
s\frac{\lambda_{H,j}t_K^j}{j!}=s^2\lambda_{K,i}\lambda_{H,j}\binom{i+j}{j}\frac{t_K^it_H^j}{(i+j)!}+s\lambda_{K,i}\lambda_{H,j}\binom{i+j-1}{j}\frac{t_K^{i-1}t_H^{j}}{(i+j-1)!}.
\]
Therefore the exponential series that computes the number of forests whose decomposition is given by one $K$-tree and one $H$-tree is:
\begin{equation}
\label{casobase}
\lambda_{H,K}=s_H\lambda_H(t_H)s\lambda_K(t_K)
\end{equation}
where $s_H=s+\frac{\partial}{\partial t_K}$ and   $s$  again  counts the number of connected components. More precisely,   the coefficient of the monomial \(s^k\frac{t_K^it_H^j}{(i+j)!}\)  is the number of forests with $k$ trees,  \(i\) leaves descending from vertices labelled by \(H\) and  $j$ leaves descending from vertices labelled by \(K\). 

\medskip
Let $\mathcal{F}'(n,G,V)$ be  the subset of $\mathcal{F}(n,G,V)$ given by the forests whose decomposition does not contain \(G\)-trees and let $\mathcal{F}''(n,G,V)$ be  the subset of $\mathcal{F}'(n,G,V)$ given by the forests whose decomposition does not contain fallen leaves.    Putting  \(\mathcal H=\mathcal K-\{G\}\)  we   denote by
\begin{enumerate}
\item \(\gamma''(j, (a_H)_{H\in \mathcal H})\)   the number of forests in $\mathcal{F}''(n,G,V)$  that have \(j\) connected components and     $k=\sum_{H\in \mathcal H}a_H$ leaves such that, for every \(H\in \mathcal H\) there are \(a_H\) leaves attached to \(H\)-trees;

\item \(\gamma'(h, j, (a_H)_{H\in \mathcal H})\)   the number of forests in $\mathcal{F}'(n,G,V)$  that have \(h\geq 0\) fallen leaves, \(j\) (with \( j> h\)) connected components and     $k=\sum_{H\in \mathcal H}a_H$ leaves such that, for every \(H\in \mathcal H\) there are \(a_H\) leaves attached to \(H\)-trees.

\end{enumerate}
 We will consider the following series:
\[\tilde{\Gamma}= \tilde{\Gamma}(s, (t_H)_{H\in \mathcal H})
=1+\sum_{
\begin{array}{c}
   j, (a_H)_{H\in \mathcal H} \\

\hbox {s.t.} \ \sum_{H\in \mathcal H}a_H\geq 1
\end{array}} \gamma''(j,(a_H)_{H\in \mathcal H})s^j\frac{\prod_{H\in \mathcal H}t_H^{a_H}}{(\sum_{H\in \mathcal H}a_H)!}\]
\[ \overline{\Gamma}= \overline{\Gamma}(s, t, ( t_H)_{H\in \mathcal H})
=\sum_{
\begin{array}{c}
   h, j, (a_H)_{H\in \mathcal H} \\

 j>h\geq 0 \\
 \sum_{H\in \mathcal H}a_H\geq 1
\end{array}} \gamma'(h, j,(a_H)_{H\in \mathcal H})s^j\frac{t^h\prod_{H\in \mathcal H}t_H^{a_H}}{(h+\sum_{H\in \mathcal H}a_H)!}.\]

\begin{teo}
\label{teogstgenerale}
We have
\begin{equation}
\label{primaformula}\tilde{\Gamma}(s,(t_H)_{H\in\mathcal{H}})=\prod_{H\in\mathcal{H}}e^{s_H\lambda_{H}{(t_H)}}
\end{equation}\begin{equation}
\label{secondaformula} \overline{\Gamma}(s, t, ( t_H)_{H\in \mathcal H})=e^{st}(\tilde{\Gamma}-1)
\end{equation}
where for every \(H\in \mathcal H\) we put ${\displaystyle s_H=(s+\sum_{
\begin{array}{c}
   K\in\mathcal{H} \\

H\subsetneq K
\end{array}}\frac{\partial}{\partial t_K})}$.

\end{teo}

\begin{proof} Given a forest \(F\in \mathcal{F}'(n,G,V)\) let us denote by \(\mathcal H_F\) the subset of \(\mathcal H\) given by all the subgroups \(H\) such that in the decomposition of \(F\) there is a \(H\)-tree. The proof of the formula (\ref{primaformula}) is    by induction on the following proposition: {\em the formula (\ref{primaformula})  enumerates correctly the forests \(F\) such that  \(|\mathcal H_F|=m\)}.

The case \(m=1\) is trivial.
Suppose that the proposition holds for \(m\geq 1\) and let \(S\) be a forest with   \(|\mathcal H_S|=m+1\). Let \(H\) be a minimal element (with respect to inclusion) of  \(\mathcal H_S\). Then we can think of \(S\) as obtained from  a forest \(S'\)  with \(|\mathcal H_{S'}|=m\) by adding some new connected components made by \(H\) trees and/or by attaching some \(H\) trees to some of the leaves.
The formula (\ref{casobase}) indicates  how to count all the forests obtained  in this way: the only difference with that simpler example is that now we use ${\displaystyle s_H=(s+\sum_{
\begin{array}{c}
   K\in\mathcal{H} \\

H\subsetneq K
\end{array}}\frac{\partial}{\partial t_K})}$, where every formal differential \(\frac{\partial}{\partial t_K}\) takes into account the attachments of the \(H\)-tree to a leaf of  a \(K\)-tree, with \(H\subsetneq K\).

In order to prove formula (\ref{secondaformula}) we need to take into account  the fallen leaves. These are labelled leaves  disconnected from the rest of the graph, so we have    $\overline{\Gamma}(s, t, ( t_H)_{H\in \mathcal H})=e^{st}\tilde{\Gamma}-e^{st}=e^{st}(\tilde{\Gamma}-1)$. We removed from the computation (by subtracting \(e^{st}\)) the contribution of the forests whose only components are fallen leaves, since    they do not belong to $\mathcal{F}'(n,G,V)$ (they do not represent a nested set). \end{proof}

Now we want to take into account  the forests \(F\in \mathcal{F}(n,G,V)\) such that $G\in {\mathcal H}_F$.
Let \(\gamma(j, n)\)   be the number of forests in $\mathcal{F}(n,G,V)$  that have  \(j\) (with \( j> 0\)) connected components and     $n>0$ leaves.
We want to give a formula for the series
\[ \mathcal{G}(s,t)= \sum_{j>0, n>0}\gamma(j, n)s^j\frac{t^n}{n!}. \]

We start recalling  that given a forest \(F\in \mathcal{F}(n,G,V)\), no more than one tree of \(F\) can have $G$-labelled internal vertices and that the $G$-labelled vertices form a path in the oriented tree.

Then we denote by $\tilde \Gamma(s,t)$ (resp. $\overline{\Gamma(s,t)}$) the  series $\tilde{\Gamma}(s,(t_H)_{H\in\mathcal{H}}$)  (resp. $\overline{\Gamma}(s,(t_H)_{H\in\mathcal{H}})$) evaluated  in $t_K=t$, for all $K\in\mathcal{H}$.
We also put  $\Gamma(s,t)=e^{st}\tilde{\Gamma}(s,t)$.

Let us  now count for instance the number of trees in $\mathcal{F}(n,G,V)$ with two  $G$-vertices $\alpha$ and $\beta$ (see Figure \ref{fig}).

\begin{figure}[H]
\centering
  \begin{tikzpicture}[scale=.8]
     \draw (3,3.1) node[anchor=east]  {$\beta$};
     \draw (5.5,4.6) node[anchor=east]  {$\alpha$};
     \draw (0.5,0.2) node[anchor=north]  {$\ldots$};
     \draw (3.5,0.2) node[anchor=north]  {$\ldots$};
     \draw (7.5,0.2) node[anchor=north]  {$\ldots$};
    \foreach \x in {1.5}
    \draw[xshift=\x cm,thick, black,fill=white] (\x cm,3) circle (0.1);
    \foreach \x in {2.75}
    \draw[xshift=\x cm,thick, black,fill=white] (\x cm,4.5) circle (0.1);
\foreach \Point in {(0,0), (1,0),(3,0),(4,0), (7,0) , (8,0),(9.5,0),(10.5,0),(0.5,1.5),(3.5,1.5),(7.5,1.5),(8,3)}{
    \node at \Point {\textbullet};
}
    \draw[thick] (0.5,1.5) -- +(-0.5,-1.5);
    \draw[thick] (0.5,1.5) -- +(+0.5,-1.5);
    \draw[thick,dashed] (0.5,1.5) -- +(+0.2,-1.5);
    \draw[thick,dashed] (0.5,1.5) -- +(-0.2,-1.5);
    \draw[thick] (0.5,1.5) -- +(2.5,1.5);
    \draw[thick] (3.5,1.5) -- +(-0.5,-1.5);
    \draw[thick] (3.5,1.5) -- +(+0.5,-1.5);
    \draw[thick] (3.5,1.5) -- +(-0.5,1.5);
    \draw[thick,dashed] (3.5,1.5) -- +(+0.2,-1.5);
    \draw[thick,dashed] (3.5,1.5) -- +(-0.2,-1.5);
    \draw[thick] (7.5,1.5) -- +(-0.5,-1.5);
    \draw[thick] (7.5,1.5) -- +(+0.5,-1.5);
    \draw[thick] (7.5,1.5) -- +(0.5,1.5);
     \draw[thick,dashed] (7.5,1.5) -- +(+0.2,-1.5);
    \draw[thick,dashed] (7.5,1.5) -- +(-0.2,-1.5);
     \draw[thick] (8,3) -- +(1.5,-3);
    \draw[thick] (8,3) -- +(2.5,-3);
    \draw[thick] (5.5,4.5) -- +(2.5,-1.5);
    \draw[thick] (5.5,4.5) -- +(-2.5,-1.5);
\draw [thick,decorate,decoration={brace,amplitude=4.4pt,mirror},xshift=0.4pt,yshift=-0.4pt](-0.2,-0.2) -- (4.2,-0.2) node[black,midway,yshift=-0.6cm] {$p$ leaves};
\draw [thick,decorate,decoration={brace,amplitude=3.9pt,mirror},xshift=0.4pt,yshift=-0.4pt](6.8,-0.2) -- (10.7,-0.2) node[black,midway,yshift=-0.6cm] {$k$ leaves};
  \end{tikzpicture}
  \caption{}
   \label{fig}
  \end{figure}

As a first step we imagine  to delete from the graph all the edges stemming from $\alpha$ or $\beta$. Then let us consider the  following two subforests (with no \(G\)-vertices)  which we suppose to be nonempty: the one made by the trees  whose roots were connected by an edge to $\alpha$ and the one  made by  the  trees whose roots were connected by an edge to $\beta$. We will call these respectively the higher forest and the lower forest and we suppose that they have respectively   $k>0$ and $p>0$ leaves.

The labels for the leaves in the higher forest can be chosen in $\binom{k+p}{k}$ ways.

The number of possible higher forests is the coefficient \(a\) of \(\frac{t^k}{k!}\) in $\Gamma (1,t)$ and the number of possible  lower forests is the coefficient \(b\) of \(\frac{t^p}{p!}\). Notice  that we are evaluating $\Gamma(s,t)$ in $s=1$ since in this computation  the number of connected components is not relevant (we are subdividing  a tree). The number of trees described in the example is therefore
\[
\binom{k+p}{k}ab.
\]

An immediate  generalization of the same reasoning shows  that  the exponential generating series for the cardinality  of the set of trees with $l$ $G$-vertices is computed by the series $(\Gamma(1,t)-1)^l$.
We are now ready to compute the exponential series that enumerates  the  trees with at least one $G$-vertex. This is: 
\[
\Phi(t)=(\Gamma(1,t)-1)+(\Gamma(1,t)-1)^2+\ldots+(\Gamma(1,t)-1)^n\ldots=\frac{1}{2-\Gamma(1,t)}-1.
\]

Now, since  a forest has at most one connected component that contains \(G\)- vertices, we obtain: 
\begin{teo} 
\label{teogst}The following formula for \(\mathcal{G}(s,t)\) holds:
\[\mathcal{G}(s,t)=s(\frac{1}{2-\Gamma(1,t)}-1)\Gamma(s,t)+\overline{\Gamma(s,t)}.\]
The number of nested sets in $L_{G,n}$ is the coefficient of $\frac{t^n}{n!}$ in $\mathcal{G}(1,t)$.
\end{teo}

\begin{obs}
 A  reader who is familiar with  the  theory of species of structures may have noticed that, 
  given a finite group representation $\rho:G\rightarrow GL(V)$ and a set of cardinality $n$, the family of forests $\mathcal{F}(n,G,V)$ is a species of structure. 
 From this point of view  our computation is in the same spirit of the results  in Chapter 3 of \cite{Berg}.
\end{obs}

\section{A formula for the  \(H\)-trees}
\label{secHtrees}
We conclude our  computation of \(\mathcal{G}(s,t)\) by providing  formulas for the series $\lambda_{H}(t_H)$, for every closed subgroup \(H\neq G\), in the case when \(G\) is abelian. For brevity of notation in this section  we will write  \(t\) instead than \(t_H\).

\begin{teo}
\label{formulaHtrees}
We have 
\[\lambda_H(t)=t+\int \left[\left (  \prod_{i\geq 2} e^{\frac{\frac{\partial}{\partial t}}{r}\frac{(2rt)^i}{i!}}\right )-1\right]\]

\end{teo}
\begin{proof}
We first find a formula for the series $\lambda_{\{e\}}(t)=\lambda(t)$ mimicking  the strategy used in Theorem 5.1 of \cite{Ga1}. 

The $\{e\}-$trees are rooted, oriented  trees with labelled leaves and labelled edges. Every internal vertex different from the root is connected to at least three other vertices, while the root of the tree must have at least two outgoing edges.
If there are \(n\) leaves, their  labels are the elements of \(\{1,...,n\}\). 

 The labels of the edges are cosets of the subgroup $\{e\}$, i.e. elements of $G$. According to Lemma $\ref{lemav1}$, for every internal  vertex, we can consider one of its outgoing edges to be labelled with $e$.
 In particular we put the edge that connects the vertex with the subtree containing the smallest leaf to be always labelled with $e$. 
 We will denote the set of trees described above as $\mathcal{T}_{\{e\}}$.

 Our first goal is to put $\mathcal{T}_{\{e\}}$ in bijection  with a family of weighted partitions. This can be done using a slight modification of the bijection in  Theorem 2.1 of \cite{Ga3}. In the following  example we show  how a labelled partition is associated to a tree in $\mathcal{T}_{\{e\}}$.
 \begin{example}
Let us  consider the  tree   in 
\ref{figbis} (a). We label its internal vertices with the elements of  the set of integers $s=\{7,8,9, 10\}$,  as in picture \ref{figbis} (b) (see the green labels). The criterion we used is the following one: we associate to every internal vertex \(v\) the set (of the labels) of the leaves \(L_v\) of the subgraph stemming out of it. Then the label of a vertex \(v\) is less than the label of a vertex \(w\) if and only if either  \(L_v\) is included into \(L_w\) or the two sets are disjoint and \(min \ L_v<min \ L_w\). 

We are now ready to associate to the tree the  labelled partition obtained considering, for every internal vertex, the labels of the vertices connected to it by an outgoing  edge:
\[
\{1^e,2^{g'},3^g\}\{4^e,6^h\}\{5^k,8^e\}\{7^e,9^t\}.
\]
Notice that the number of elements  in the underlying set is 9, which is equal to the number of leaves of the tree plus the number of its internal vertices minus 1.
 \end{example}

\begin{figure}
\centering
  \begin{tikzpicture}[scale=.8]
     \draw (0,0) node[anchor=east]  {1};
     \draw (1,0) node[anchor=east]  {3};
     \draw (2,0) node[anchor=east]  {2};
     \draw (4,0) node[anchor=east]  {4};
     \draw (5.6,0) node[anchor=east]  {6};
     \draw (3.2,0) node[anchor=east]  {5};
     \draw (1,.7) node[anchor=east]  {g};     
     \draw (1.5,.75) node[anchor=west]  {g'};  
     \draw (1.83,2.4) node[anchor=east]  {e};
     \draw (0.5,.75) node[anchor=east]  {e};     
     \draw (3.25,2.95) node[anchor=west]  {t};
     \draw (3.6,1.3) node[anchor=east]  {k};
     \draw (4.4,2) node[anchor=west]  {e};
     \draw (4.4,.75) node[anchor=east]  {e};
     \draw (5.2,.75) node[anchor=west]  {h};
     
    \draw (2.75,-0.7) node[anchor=north]  {\textbf{(a)}};
    \draw (10.75,-0.7) node[anchor=north]  {\textbf{(b)}};         
     
     \draw (8,0) node[anchor=east]  {1};
     \draw (9,0) node[anchor=east]  {3};
     \draw (10,0) node[anchor=east]  {2};
     \draw (12,0) node[anchor=east]  {4};
     \draw (13.6,0) node[anchor=east]  {6};
     \draw (11.2,0) node[anchor=east]  {5};
     \draw (9,.7) node[anchor=east]  {g};     
     \draw (9.5,.75) node[anchor=west]  {g'};  
     \draw (9.83,2.4) node[anchor=east]  {e};
     \draw (8.5,.75) node[anchor=east]  {e};     
     \draw (11.25,2.95) node[anchor=west]  {t};
     \draw (11.6,1.3) node[anchor=east]  {k};
     \draw (12.4,2) node[anchor=west]  {e};
     \draw (12.4,.75) node[anchor=east]  {e};
     \draw (13.2,.75) node[anchor=west]  {h};     
     \draw (12,2.5) node[anchor=west]  {\textcolor{green}{9} };
    \draw (12.8,1.5) node[anchor=east]  {\textcolor{green}{8}};
     \draw (9,1.5) node[anchor=east]  {\textcolor{green}{7}};
     \draw (10.5,3.2) node[anchor=east]  {\textcolor{green}{10}};
\foreach \Point in {(8,0), (9,0),(10,0),(9,1.5), (11.2,0) , (12,2.5),(12.8,1.5),(12,0),(13.6,0), (10.5,3.2),
(0,0), (1,0),(2,0),(1,1.5), (3.2,0) , (4,2.5),(4.8,1.5),(4,0),(5.6,0), (2.5,3.2)}{
    \node at \Point {\textbullet};
}
    \draw[thick] (1,1.5) -- +(-1,-1.5);
    \draw[thick] (1,1.5) -- +(0,-1.5);
    \draw[thick] (1,1.5) -- +(1,-1.5);
    \draw[thick] (1,1.5) -- +(1.5,1.7);
    \draw[thick] (4,2.5) -- +(-0.8,-2.5);
    \draw[thick] (4,2.5) -- +(+0.8,-1);
    \draw[thick] (4.8,1.5) -- +(-0.8,-1.5);
    \draw[thick] (4.8,1.5) -- +(+0.8,-1.5);
    \draw[thick] (2.5,3.2) -- +(1.5,-0.7);
    \draw[thick] (9,1.5) -- +(-1,-1.5);
    \draw[thick] (9,1.5) -- +(0,-1.5);
    \draw[thick] (9,1.5) -- +(1,-1.5);
    \draw[thick] (9,1.5) -- +(1.5,1.7);
    \draw[thick] (12,2.5) -- +(-0.8,-2.5);
    \draw[thick] (12,2.5) -- +(+0.8,-1);
    \draw[thick] (12.8,1.5) -- +(-0.8,-1.5);
    \draw[thick] (12.8,1.5) -- +(+0.8,-1.5);
    \draw[thick] (10.5,3.2) -- +(1.5,-0.7);    

    \draw[thick] (-1,4) -- +(7.5,0);   
    \draw[thick] (-1,-0.5) -- +(7.5,-0);   
    \draw[thick] (-1,-0.5) -- +(0,4.5);   
    \draw[thick] (6.5,-0.5) -- +(0,4.5);       

    \draw[thick] (7,4) -- +(7.5,0);   
    \draw[thick] (7,-0.5) -- +(7.5,-0);   
    \draw[thick] (7,-0.5) -- +(0,4.5);   
    \draw[thick] (14.5,-0.5) -- +(0,4.5);
  \end{tikzpicture}
\caption{} 
\label{figbis} 
\end{figure}

Therefore we have to count partitions instead than trees. Let us denote by \(\mathcal P(n,k)\) the set of the above mentioned partitions with \(k\) parts and whose underlying set has cardinality \(n\). Then we  put \(p(t,z)=\sum_{n\geq 2, k\geq 1} |\mathcal P(n,k)|t^nz^k  \).

 We first  count the contributions to \(p(t,z)\)  provided the parts with  fixed cardinality \(i\geq 2\). If in a partition there is only one part of cardinality \(i\) its contribution to \(p(t,z)\)  is $(\frac{z}{r})(\frac{(rt)^i}{i!})$, where    \(r=|G|\); if there are \(j\) parts 
their contribution  is $(\frac{z}{r})^j\frac{(\frac{(rt)^i}{i!})^j}{j!}$.

Then putting together all the sizes \(i\geq 2\) we obtain    the  formula  \[ p(t,z)= \left ( \prod_{i\geq 2} e^{\frac{z}{r}\frac{(rt)^i}{i!}}\right ) -1 \]
 
 Now we observe that a tree with \(n\) leaves contributes to the coefficient of \(\frac{t^n}{n!}\)  in the exponential series \(\lambda(t)\), so  to obtain a formula for \(\lambda(t)\) it suffices to put  \(z=\frac{\partial}{\partial t}\) in the formula above for \(p(t,z)\) and perform a final integration:

\[
\lambda(t)= \int\left [ \left ( \prod_{i\geq 2} e^{\frac{\frac{\partial}{\partial t}}{r}\frac{(rt)^i}{i!}}\right )-1 \right ]
\]

%
Now we can compute the series $\lambda_H(t)$ for all the  closed subgroups \(H\) different from \(\{e\}\) and \(\{G\}\). Unlike the case $H=\{e\}$ there may be internal vertices with only one outgoing edge. In particular the trees must satisfy:
\begin{itemize}\label{alberi}
  \item If an internal vertex is connected to only two vertices, then one of these two vertices must be a leaf. We will call these vertices {\em one leaf vertices}. We can also have the case of the {\em one leaf tree}, with only the root and one leaf.
  \item All the edges are labelled with cosets of $H$. Among all the edges stemming from a node, the one connecting the vertex to the subtree with the smallest leaf is labelled with $eH$.
\end{itemize}

First we consider the set $\mathcal{T}_H$ of $H$-trees with no one leaf vertices and we compute its exponential series $\overline{\lambda}_H(t)$. Now we observe that if we forget the labels of the edges, the only difference between a tree in $\mathcal{T}_H$ and a tree in $\mathcal{T}_{\{e\}}$ is that the labels of the vertices are all equal to \(H\) instead than being all equal to \(\{e\}\).
 
So as in the case $H=\{e\}$ we can put $\mathcal{T}_H$ in bijection with a set of partitions. 
We obtain partitions of the type 
\(\{1^{eH},2^{g'H},3^{gH}\}\{4^{eH},6^{hH}\}\{5^{kH},8^{eH}\}\{7^{eH},9^{tH}\}\): the only difference  with the $\mathcal{T}_{\{e\}}$ case is that now  the labels are provided by the labels  of the edges of the trees in $\mathcal{T}_H$, so they belong to  $G/H$. 
This allows us to write
\[
\overline{\lambda}_H(t)= \int \left[\left (  \prod_{i\geq 2} e^{\frac{\frac{\partial}{\partial t}}{r}\frac{(rt)^i}{i!}}\right )-1\right ]
\]
where now $r=|\frac{G}{H}|$.
To count the number of $H-$trees we observe that  each $H-$tree can be seen  as a tree in  $\mathcal{T}_{H}$  with possibly some one leaf vertices glued to some of the leaves.
From a tree in  $\mathcal{T}_{H}$ with \(i\) leaves we can obtain \(\binom {i}{k}\) different \(H\)-trees with exactly \(k\) one leaf vertices. So a tree in  $\mathcal{T}_{H}$ with \(i\) leaves produces $\sum_{k=0}^{i}\binom{i}{k}=2^i$ different $H-$trees. Assuming that $\overline{\lambda}_H(t)=\sum_{i\geq 2} \frac{a_it^i}{i!}$, we have that 
\[
\lambda_H(t)=t+\sum_{i\geq 2} \frac{a_it^i}{i!}2^i=t+\overline{\lambda}_{H}(2t).
\]
This gives the formula in the statement (the addendum  $t$ appears to take into account   the one leaf tree).
\end{proof}
\begin{example}
Let $G=\mathbb{Z}/2\mathbb{Z}\times \mathbb{Z}/2\mathbb{Z}$ and $\rho$ the following faithful group action:
\begin{align*}
\rho:G&\longrightarrow GL(\mathbb{C}^2)\\
(1,0)&\longrightarrow  \begin{bmatrix}
-1 & 0 \\
0 & 1 
\end{bmatrix}\\
(0,1)&\longrightarrow  \begin{bmatrix}
1& 0 \\
0 & -1
\end{bmatrix}
\end{align*}
The closed subgroups associated to $\rho$ are $\{e\}$, $H_1=<(1,0)>$, $H_{2}=<(0,1)>$ and $G$, in particular $<(1,1)>$ is not a closed subgroup. Therefore we have  three different series: $\lambda_{\{e\}}(t)$, $\lambda_{H_1}(t)$ and $\lambda_{H_2}(t)$. We notice that $\lambda_{H_2}(t)=\lambda_{H_1}(t)$ because they have the same parameter $r=2$ while for $\lambda_{\{e\}}(t)$ $r=4$. 
 Then  $\tilde{\Gamma}(s,t_{\{e\}}, t_{H_1},t_{H_2})=e^{s_{\{e\}}\lambda_{\{e\}}{(t_{\{e\}})}}e^{s_{H_1}\lambda_{H_1}{(t_{H_1})}}e^{s_{H_2}\lambda_{H_2}{(t_{H_2})}}$ and  $\tilde{\Gamma}(s,t)$ is its evaluation obtained putting  \(t=t_{\{e\}}=t_{H_1}=t_{H_2}\). Thus \(\Gamma(s,t)= e^{st} \tilde{\Gamma}(s,t)\) is the main ingredient in the formula of Theorem \ref{teogst}.
\end{example}

\addcontentsline{toc}{section}{References}
\bibliographystyle{acm}
\bibliography{Bibliogpre}

\end{document}